\documentclass[12pt]{amsart}
\usepackage{amsmath}
\usepackage{amssymb}
\usepackage{amsfonts}
\usepackage{amsthm}
\usepackage{verbatim}
\usepackage{amscd}
\usepackage{cite}
\usepackage{leftidx}
\usepackage{enumerate}
\usepackage{txfonts}
\usepackage{manfnt}
\usepackage{amscd}
\usepackage[mathscr]{eucal}
\usepackage{hyperref}
\usepackage{datetime2}
\usepackage{mathpazo}
\def\Ord{\mathrm{Ord}}

\def\orp{\mathrm{orp}}

\newtheorem{thm}{Theorem}

\newtheorem*{thm*}{Theorem}
\newtheorem*{prop*}{Proposition}
\newtheorem{cor}[thm]{Corollary}
\newtheorem*{cor*}{Corollary}

\newtheorem{lem}[thm]{Lemma}
\newtheorem*{lem*}{Lemma}

\newtheorem*{claim*}{Claim}

\theoremstyle{remark}

\newtheorem{rem}[thm]{Remark}
\newtheorem*{rem*}{Remark}
\newtheorem*{incirem*}{Incidental remark}
\newtheorem{crit-rem}[thm]{Critical remark}

\newtheorem*{remarks*}{Remarks}

\newtheorem{example}[thm]{Example}
\newtheorem*{example*}{Example}

\newtheorem*{defn*}{Definition}

 \DeclareMathOperator{\Hom}{Hom}

\def\refp #1.{(\ref{#1})}

\newcommand{\A}{\mathcal{A}}

\newcommand{\ul}[1]{\underline {#1}}

\def\sbr #1.{^{[#1]}}
\def\sfl #1.{^{\lfloor #1\rfloor}}

\newcommand\red{{\mathrm red}}

\def\?{{\bf{??}}}

\def\A{\Bbb A}

\def\C{\mathbb C}
\def\P{\mathbb P}

\def\R{\mathbb R}

\def\ord{\text{\rm ord}}

\def\Spec{\text{\rm Spec} }

\def\O{\mathcal O}

\def\g{\mathfrak g}

\def\m{\mathfrak m}

\def\1/2{\frac{1}{2}}

\def\I{\mathcal{ I}}

\def\2{{[2]}}
\def\l{\ell}
\def\nl{\newline}
\def\n{\mathcal{N}}

\def\<{\langle}
\def\>{\rangle}

\def\2{{[2]}}
\def\l{\ell}

\def\scl #1.{^{\lceil#1\rceil}}
\def\spr #1.{^{(#1)}}
\def\sbc #1.{^{\{#1\}}}

\def\subpr#1.{_{(#1)}}

\def\beq{\begin{equation*}}
\def\eeq{\end{equation*}}

\def\g3{{\Gamma\spr 3.}}
\def\gg{{\Gamma\spr 2.}}

\def\gg{\mathbb G}

\newcommand{\eqspl}[2]{
\begin{equation}\label{#1}
\begin{split}
#2\end{split}\end{equation}}
\newcommand{\eqsp}[1]{\begin{equation*}
\begin{split}#1\end{split}\end{equation*}}

\newcommand{\exseq}[3]{
0\to #1\to #2\to #3\to 0
}

\newcommand{\beginalphaenum}{
\begin{enumerate}\renewcommand{\labelenumi}{ }
\item \begin{enumerate}
}

\def\eex{\end{rm}\end{example}}

\begin{document} 

\title{On the size and local equations \\ of
Fibres of general projections
}
\author 
{Ziv Ran}


\thanks{arxiv.org}
\date{\DTMnow}


\address {\nl UC Math Dept. \nl
 Skye surge facility, Aberdeen-Inverness Road,
\nl
Riverside CA 92521 US\nl 
ziv.ran @  ucr.edu\nl
\url{http://math.ucr.edu/~ziv/}
}

 \subjclass[2010]{14N05, 14N15}
\keywords{generic projections, multiple points, singularities}

\begin{abstract}
	For a general birational projection of a smooth nondegenerate projective
	$n$-fold from $\P^{n+c}$ to $\P^m$, $n<m\leq(n+c)/2$, all fibres have total length asymptotically bounded by
	$2^{\sqrt{n}+1} $ and the fibres are locally defined by linear and quadratic equations.
	
\end{abstract}
\maketitle
\section*{Introduction}
Let $X$ be a smooth variety of dimension $n$ in $P=\P^{n+c}$ over $\C$. Let $\Lambda$
 be a general linear subspace of dimension $\lambda<c$ in $P$,
 perforce disjoint from $X$, and let
 \[\pi:=\pi_\Lambda: X\to Q:=\P^m, m:={n+c-\lambda-1}\]
 be projection from $\Lambda$ restricted on $X$. 
 Elements of $Q$ are viewed as $(\lambda+1)$-dimensional
  linear subspaces  $L\subset P$ containing $\Lambda$.
 Let
 \[X^\Lambda_k\subset Q\]
 be the $k$-fold locus of $\pi$, i.e.  the locus in 
 $Q$ of fibres of $\pi$ that have length $k$ or more.
 Note that it is $m-n=c-\lambda-1$ conditions for $L$
 to meet $X$, hence $k(m-n)$ conditions for $L$
 to meet $X$ $k$ times; thus the expected codimension of $X^\Lambda_k$ in
 $Q$ is $k(m-n)$.\par
 The study of the projections $\pi_\Lambda$,  
 their fibres and the loci $X^\Lambda_k$ has a long history in
 classical through modern Algebraic Geometry (some of which is reviewed in
 \cite{grp}, \cite{beis} and \cite{griffiths-harris}). 
 In the case of small $\lambda$, a real breakthrough in their
 modern study was obtained fairly recently by Gruson and Peskine
 \cite{grp}, who gave complete results in the
 case where $\Lambda$ is a point. An alternate proof of the
 Gruson-Peskine theorem, and some partial extensions were
 given in \cite{filling}, \cite{fillingreduced}. In particular, it was
 shown in \cite{fillingreduced} in the case where $\Lambda$ is
 a line,  that a generic fibre of given
 length is reduced, i.e. a collection of distinct points.\par
 The case of projection from a higher-dimensional center has
 remained more mysterious. Focusing to fix ideas on the case of 
  projection of an $n$-fold to $\P^{n+1}$,
 Mather's work \cite{mather}, as carried over to the setting of complex
 algebraic geometry by Alzati and Ottaviani \cite{alzati-ottaviani}
 and summarized in \cite{beis}, shows there  that the projection has 
 corank at most $\sqrt{n}$ at any point
 and that the number of distinct points in a fibre is at most $n$.
 For projections to higher $\P^m$ there are better bounds.
 A construction of Lazarsfeld \cite{laz-pos}
 shows that there exist many examples of generic projections to $\P^{n+1}$ with
 points of corank $\sqrt{n}$,
 asymptotically the largest
  possible, and consequently, as shown by Beheshti-Eisenbud \cite{beis},
   such projections have schematic  fibres
 of length that grows exponentially with $\sqrt{n}$
 (more precisely the length is asymptotically 
 at least $\frac{\sqrt{2/\pi}}{\sqrt[4]{n}}2^{\sqrt{n}}$ for $n>>0$).  
 Such fibres are automatically obstructed . 
 The paper \cite{beis} also gives a certain
 bound on an invariant related to the fibre, extending some work of Mather \cite{mather}.\par
 In this paper we will prove an asymptotically sharp bound on the fibre length
 of (birational) 
 general projections, which shows in particular that Lazarsfeld's examples are
 essentially worst possible; namely, we will prove the following
 (see Corollary \ref{order2} and Corollary \ref{exp}):
 \begin{thm*}
 	Let $X$ be a smooth nondegenerate $n$-fold in $\P^{n+c}$ and let
 	$n<m\leq\frac{n+c}{2}$. Then for the general projection of $X$ to $\P^m$,
 	all fibres have total length at most $2\max(2^{\sqrt{n}}+n-1, 1+2\sqrt{2}(n-1))$ and are locally
 	at each point defined by linear and quadratic equations.
 	\end{thm*}

This result comes about as follows. We introduce a numerical measure
called \emph{order sequence}, different than 
intersection length
but related to it,
and pertaining to the contact or relative position
of a subvariety $X$ and a linear space $L$ at an
isolated point of their intersection, or
more generally the relative position on an ambient variety 
(such as $\P^{n+c}$) of a subscheme
(such as $X$) and a linear series (such as the series of hyperplanes
through $L$). The basic properties of order sequences are developed
in \S\ref{order-general}, \S\ref{order-projective} and
\S\ref{order-subsheaf}.
Our main technical result is an upper bound on this measure
(Theorem \ref{contact-bound}).
  It is this upper bound that yields Corollary \ref{order2})
  and Corollary \ref{exp}). As for the upper bound, it 
   is a consequence
  of the first-order deformation theory associated to the order 
  sequence, plus a 'vanishing lemma' which already appeared in \cite{fillingreduced}
  and which in turn is ultimately a reflection of 'generic smoothness'
  in char. 0.\par
  I am grateful to the referee for a great many  corrections and helpful comments.
\section{Setup}  
We work over $\C$.
To state our results we need some preliminaries.
First,  notation-wise, for a point $p$ on an $n$-dimensional 
subvariety $X\subset\P^{n+c}$, $T_pX$ denotes the Zariski
  tangent space; for $Z$ is a finite-length scheme, $\l_p(Z)$ denotes local length at $p$.
  For a linear space $\Lambda$ of dimension $\lambda$ disjoint
  from $X$ we let $Q=\P^m$ denote the
  target space for the projection $\pi$ from $\Lambda$, i.e. the set of linear subspaces $L$
  of dimension $\lambda+1$ containing $\Lambda$.
  For $L\in Q$
  and $p\in L\cap X$, set
  \[e_p=n+c-\dim(T_pX+T_pL).\]
  Note $e_p$ is just the corank of the projection $\pi$ (on $X$) at $p$, i.e. $m-\dim(d\pi_p(T_pX))$.
  Also $e_p=\dim (\check N_{L, p}\cap \check N_{X, p})$ where $\check N_{L, p}$
  denotes the fibre of the conormal bundle in $\P^{n+c}$ to $L$ at $p$
  and ditto for $X$.\par
\section{Order sequence and filtration: general case}\label{order-general}  
The statement involves the notion of \emph{order sequence} $(d_\bullet)$. We define it here is greater
generality than needed for our application here.\par
  This is associated to the following data: 
  \begin{itemize}\item
  an ambient space $P$, 
  \item a point
  $p \in P$, 
  \item a linear system, 
  i.e. a finite-dimensional subspace $V\subset H^0(A)$ for some line
  bundle $A$ over $P$, with evaluation map $e:V\otimes\O_P\to A$,
  \item an ideal $\I\subset \O_P$,
  such that $\I A+e(V)\O_P\subset\O_P$ contains $\m_p^NA$
  for $N>>0$. 
  \end{itemize}
 ( The only case needed in this paper is where
  $P$ is a projective space, $A=\O(1)$, $V$ is the system of
  hyperplanes containing a linear subspace $L\subset P$ and $\I$ is the ideal of a
  variety $X\subset P$ ).\par
  Back to the general case, we denote by $L$ and $X$ respectively the base scheme
  of $V$, whose ideal is $\mathrm{bs}(V)$, i.e.
   the image of $V\otimes\O_P(-A)\to\O_P$,
  and the subscheme of $P$ corresponding to $\I$, so that
  $\O_X=\O_P/\I$. Thus, we are assuming
  that $p$ is an isolated point of $X\cap L$. To simplify matters we shall
  also assume that \emph{no nonzero element of $V$ vanishes on $X$}
  (i.e. that $X$ is 'nondegenerate' with respect to $V$).\par

The order sequence is  defined  as follows. 
First,
\[d_{1}=\max(\ord_p(\tilde f_{1}):
\tilde f_1\in \m_pA \mathrm{\ and\ } \exists 0\neq y_{1}\in V
\mathrm{\ such\ that\ }  y_1-\tilde f_{1}\in\I(A)_p).\]
Thus $d_1$ is the largest order at $p$
  of any nonzero element in $ e(V)+ \I(A)$; equivalently,
  the largest order at $p$ {\emph{when restricted
  		 on $X$}} of any nonzero element of $V$. We set
  	 \[f_1=\tilde f_1|_X=y_1|_X.\]
  \par
  Assuming $(d_1, y_1), ...,(d_i, y_i)$ are defined, then set
    \[d_{i+1}=\ord_p(\tilde f_{i+1})\] as the largest order
  of any element  $f_{i+1}\in e(V)+ \I(A)$ not in the 
  $\C$-span of of $\tilde f_1, ...,\tilde f_i$ modulo $\I(A)$.
  Thus there is an element $y_{i+1}$ in the restriction of $e(V)$
  but not in the $\C$-span of $y_1,...,y_i$ having order $d_{i+1}$
  on $X$, and $d_{i+1}$ is largest with this property. This defines
  $(d_{i+1}, y_{i+1})$.\par
  Thus  the sequence $(d_1, y_1), ...,(d_k, y_k)$ is defined, and $(d_1,...,d_k)$
  is called the order sequence associated to the above data.
Note that the sequence $(d_\bullet)$ is uniquely determined. It
will be denoted
$\Ord_p(V, I)$ or $\Ord_p(V, X)$. Note
  \[d_1\geq d_2\geq...\geq d_k, k=\dim(V).\]
 Also,
  $(y_1,...,y_k)$ is a basis for $V$ called an \emph{adapted} basis.
  It is not canonical, 
however if $d_r>d_{r+1}=...=d_s>d_{s+1}$ then
    the span of $y_{r+1},..,y_s$ is uniquely determined modulo
    the span of $y_{1}, ...,y_r$. 
    Consequently, if we define  the \emph{reduced}
    order sequence $\overline\Ord_p(V,X)=(\bar d_\bullet)$  as the sequence
    of distinct values of $\Ord_p(V,X)$,
  Then
    $V$ admits a canonical ascending \emph{order filtration} $F_\bullet V$
    (depending on $p$) so
    	that $F_iV$ consists of the elements of order $\geq \bar d_i$ and
    	an adapted basis is simply a basis adapted to this filtration.
 The reduced order sequence and order filtration are used in the definition of order subsheaf
 in \S \ref{order-subsheaf}.
 
  \begin{remarks*}
   These are nonessential but possibly clarifying. \par	
  \begin{enumerate}\item
  The corresponding $f_1,...,f_k$ to $y_1,...,y_k$ 
  are a set of generators $\mod\I$ - not necessarily minimal - of the ideal  $\n=\I+\mathrm{bs}(V)$,
  where $\mathrm{bs}(V)$ is the base ideal of $V$, 
  with the added property that they also generate the 'normal cone' ${\mathrm{gr}}^\bullet \n=
  \bigoplus \n\cap \m^i/\n\cap\m^{i+1}$.\item
 The sequence $\Ord_p(V,X)$ coincides
  with the 'vanishing sequence' (in the sense of Eisenbud-Harris
  \cite{eisenbud-harris})
  associated to the restriction of $V$ on $X$
  (in the case considered in \cite{eisenbud-harris}
  where $X$ is a smooth curve, the order sequence is strictly
  decreasing but in higher dimension this is not true).
  	\end{enumerate}
  \end{remarks*}
  \section{Order sequence and filtration: projective case}\label{order-projective}
  We will apply this construction in the case 
  where $P$ is a projective space,
  $L$ is a linear subspace, $V=H^0(\I_L(1))$, the hyperplanes through $L$,
  $A=\O_P(1)$, and $\I$ is the ideal
  of a subvariety $X$ (having finite intersection with $L$).\par
   In the above situation with 
   $X, L=\P^{n+c-m}\subset P=\P^{n+c}, p\in L\cap X$ isolated,
   we define the order sequence $\Ord_p(L, X)$ as
   $\Ord_p(H^0(\I_L(1)), X)$. It can be analyzed as follows.
   We will  assume $X$ is nondegenerate. We begin with the case $c\geq m$,
   i.e. $\dim(X)\leq\dim L$.
    Let $y_1,...,y_m$ be linear equations for $L$ (i.e.
    essentially, a basis for $H^0(\I_L(1))$, so $p$ is the origin.
   ( In the case where $L$ is a fibre for projection from $\Lambda$,
    these will be coordinates on the target $\P^m$ ).
    Then we may choose complementary linear coordinates 
    $x_1,...,x_{n+c-m}$- thus
    in effect choosing a general projection of $X$ unramifiedly onto
    a submanifold of  $L$-
    so that local analytic equations for $X$ in $P$ have the form
    \eqspl{x-equations}{y_i-f_{i}(x), i=1,...,m, \\
    	f_{m+1}(x),..., f_c(x)}
    where $f_1,...,f_c$ are analytic functions in $x_1,...,x_{n+c-m}$ only,
     defined near $p$.
    In other words, we represent $X$ locally as a graph over the submanifold of $L$ defined
    by $f_{m+1},...,f_c$. 
    NB: the latter projection is an auxiliary construction 
    and in the case where $L$ is a fibre of a projection $\pi$ from $\Lambda$ it
     is essentially complementary to $\pi$.\par
    Now by choosing a suitable basis for $V=H^0(\I_L(1))$
    we may assume $\ord_p(f_i))=d_i$ where $(d.)=\Ord_p(L, X)$
    so $y_1,...,y_m$ is an adapted basis.
    Note $e_p$, the corank, 
    is just the number of $f_i, i\leq m$ that have order $\geq 2$
    and that $f_{m+1},...,f_c$, being transverse defining equations for
    a submanifold, have order 1 at $p$. 
    Note also
    that the $f_{i}$ are not necessarily minimal generators for $\I_{X\cap L}$
    $\mod\I_L$ . This concludes the discussion of the case $m\leq c$.\par
    Now in case $c<m$, i.e. $\dim(X)>\dim(L)$, 
   we may similarly represent $X$ as a graph over all of $L$, and this
   yields local equations
    for $X$ of the form $y_i-f_i(x), i=1,...,c$.
    By suitably choosing $y_{c+1}, ...y_m$, we may assume
    $x_1,...,x_{n+c-m}, y_{c+1}, ...y_m$ are local coordinates on $X$
    and in particular $y_{c+1}, ...y_m$ have order 1 at $p$. 
     Then,
     replacing $y_1,...,y_c$ by an adapted basis for their
     span, we get an adapted basis
     $y_1,...,y_m$ of $H^0(\I_L(1))$ and 
     order sequence 
     $(\ord_p(f_1),...,\ord_p(f_c), 1=\ord_p(y_{c+1}),...,
    1=\ord_p(y_m))$.
    
     \begin{example}\label{3dim}
    	Here $c=2=m$. 
    	In $\A^3$ with coordinates $x, y_1, y_2$ let $L$ be the $x$-axis
    	with equations $y_1= y_2=0$ and let $X$ be
    	the curve $y_1-x^6, y_2-x^3$. Then the order sequence at the origin is
    	$(6,3)$ with adapted basis $(x^6,  x^3)$ or $(y_1, y_2)$.
    	$x$ is a coordinate on either $X$ or $L$ an in terms of this
    	the schematic intersection $X\cap L$ is defined by $x^3$.
    \end{example}
    \begin{example}\label{4dim}
    	Here $c=2<m=3$. 	In $\A^4$ with coordinates $x, y_1, y_2, y_3$ 
    	let $L$ be the $x$-axis
    	with equations $y_1= y_2=y_3=0$ and let $X$ be
    	the smooth surface $y_1-x^6, y_2-x^3$. Then note $y_3, x$ are local
    	coordinates on $X$ and the order sequence is $(6, 3, 1)$
    	with adapted basis $(y_1, y_2, y_3)$. Note the schematic
    	intersection $X\cap L$ has ideal $(y_3, x^3)$ on $X$
    	and $y_1, y_2, y_3$ are nonminimal generators for it.
    \end{example}
    \section{Order subsheaf}\label{order-subsheaf}
     The order filtration introduced in \S \ref{order-general} can be used to
    define a natural modification (full-rank subsheaf)
     at $p$ of the normal bundle $N_L$,
    denoted $N_{L,p}^\ord$ and called the
    \emph{order subsheaf}. For simplicity we do this only in the
    projective setting as in \S \ref{order-projective}, as follows. Consider
     the order filtration $0\subset F_1\subset...\subset F_r\subset V=H^0(\I_L(1))$
    introduced above and
    let
    \[F_r^\perp\subset...\subset F_1^\perp\subset V^*=H^0(\I_L(1))^*\]
    be the dual filtration. Note that $V^*\otimes\O_L(1)=N_L$.
    Let $K_1$ denote the kernel of the natural surjection
    \[N_L\to (V^*/F_1^\perp)\otimes(\O_L/\m_p^{\bar d_1-1})(1).\]
    Then let $K_2$ denote the kernel of the natural surjection
    \[K_1\to (F_1^\perp/F_2^\perp)\otimes\m_p^{\bar d_1-1}/\m_p^{\bar d_2-1}(1),\]
    etc.; then finally $N_{L,p}^\ord=K_r$. Also 
    \[N_L^\ord=\bigcap\limits_{p\in X\cap L}N^\ord_{L,p}.\]
    Also, define a subsheaf $N_L^{\ord, 1}\subset N_L^\ord$ analogously,
    replacing each $d_i-1$ by $d_i$, and similarly for the local analogues
    $N_{L, p}^{\ord, 1}$. These are used only in Remark \ref{sufficiency}.\par
\section{Results}
We continue with the above notations and for a point $p$
in the finite set $X\cap L$ we let $d_{\bullet, p}$ denote the 
order sequence at $p$.  To simplify notation, set
  \eqspl{ab}{
  a(X, L)=\sum\limits_{p\in L\cap X}\sum\limits_{i=1}^{e_p} 
  \binom{n+c-m+d_{i, p}-2}{d_{i, p}-2}
}
  \begin{thm}\label{contact-bound}
 Notations as above, assume $X$ is nondegenerate, 
 $\Lambda\subset\P^{n+c}$ is a general linear subspace of codimension
 $m+1$ disjoint from $X$. Let  $L$ be a codimension-$m$ linear space containing
 $\Lambda$ such that
 for each $p\in L\cap X$, 
 the pair $(L, p)$ is general with given order sequence $\Ord_p(L, X)$. Then
 the following \ul{Projection Inequality} holds:
  \[a(X, L)\leq m.\]
  \end{thm} 

Before starting the proof, we give some corollaries.
  Evidently, the only nonzero terms in the  sum defining $a(X, L)$ above are those with $d_{i, p}\geq 2$
  (of which there are indeed $e_p$ many by definition of corank). Those terms with $d_i=2$
  contribute 1 each while those with $d_i\geq 3$ contribute at least $n+c-m+1=\dim(L)+1$.
Thus, whenever $\dim(L)\geq m$, i.e. $n+c\geq 2m$, we have that $d_i\leq 2, \forall i$.  
In particular, the Theorem yields:
\begin{cor}\label{order2}
Assume $2m\leq n+c$ and $X$ nondegenerate. 
Then the fibre $L\cap X$ is locally defined at each point $p$ by equations of order
1 and at most $e_p$ equations of order 2.
\end{cor}
\begin{cor}\label{exp} Assumptions as above,
for a general projection to $\P^{n+1}$, 
  length of any fibre is at most $2\max(2^{\sqrt{n}}+n-1,1+2\sqrt{2}(n-1)) $
  and the local length at any point is at most $2^{\sqrt{n}}$.
  \end{cor}
  \begin{proof}[Proof of Corollary \ref{exp}] 
  	Recall that $m$, the codimension of $L$, is also the dimension of the
  	target projective space for projection from $\Lambda$.
  So here we are taking a general $\Lambda\subset\P^{n+c}$
  	of codimension $m+1=n+2$. 
Then consider an arbitrary fibre  $L_0\cap X$ 
of projection from $\Lambda$, with order sequences
$(d_{\bullet,p}=\Ord_p(L_0, X))$
  for $p\in S_0:=(L_0\cap X)_{\mathrm{\red}}$ (a finite set).
  Let $L$ be general with the property that $S:=(L\cap X)_{\mathrm{red}}$
  deforms $S_0$ and 
  $\Ord_p(L, X)=d_{\bullet,p}$ for all $p\in S$.
  Corollary \ref{order2} shows that the local length
  of $L\cap X$ at $p\in S$ is at most $2^{e_p}$. Mather's theorem \cite{mather} 
  implies that $\sum e_p^2\leq n$ and hence $|S|\leq n$
  and of course $2^{e_p}\leq 2^{\sqrt n}$.\par 
  Now in \cite{stackexchange} it is proven, using a constrained or 'bordered'-
  Hessian calculation,
  that the maximum of the 
  function $g(x)=\sum 2^{x_i}$ on the ball $B=\{(x_1,...,x_n):\sum x_i^2\leq n\}$
  in $\R^n$ is achieved when the vector $(x_1,...,x_n)$ is, up to permutation,
  of the form either \par(i) $(a,...,a)$, or\par
   (ii) $(a,...,a,b)$ with 
  $0<a<1<b$ and $2^a/a=2^b/b$, or \par 
  (iii) $(a, b,...,b)$ with $0<a<1<b$ and $2^a/a=2^b/b$.\par
  Values $g(x)$ are bounded above as follows:\par
  For $x$ of type (i), we have $a\leq 1$ because $x\in B$, hence  $g(x)\leq 2n$. \par
  For $x$ of type (ii), 
  \[g(x)\leq (n-1+2^{b-a})2^a
  \leq (n-1+2^{\sqrt{n}})2.\]
  For $x$ of type (iii) we have, using $b^2\leq n/(n-1)$ and $n\geq 3$:
  \eqsp{g(x)\leq& (1+(n-1)2^{b-a})2^a\\
  \leq& (1+(n-1)2^{\sqrt{n/(n-1)}})2<2(1+(n-1)2^{1+1/(n-1)})<2(1+(n-1)2^{3/2}).}

  It follows that in our situation 
  the max in question, with the $x_i$ nonnegative integers,
   cannot exceed
  $2\max(2^{\sqrt{n}}+n-1,1+2\sqrt{2}(n-1)) $. This implies our conclusion.
  \footnote{As the referee points out, the linear bound dominates for $n\leq 37$.}
  \end{proof}


As mentioned above, a construction due to 
Lazarsfeld \cite{laz-pos}, Vol II, Cor. 7.2.18 shows that whenever the 
cotangent bundle $\Omega_X$ is nef, then for any $\Lambda$ of codimension
$n+2$ disjoint from X 
there exist points $p\in X$ where the corank $e_p$ is roughly $\sqrt{n}$.
In these cases, Beheshti and Eisenbud \cite{beis} have shown,
using Stirling's approximation, that $\l(L\cap X)$ is
asymptotically at least $(\sqrt{2/\pi})2^{\sqrt{n}}/\sqrt[4]{n}$.
Thus the bound of Corollary \ref{exp} is 'sharp on 
the dominant term'. 
\begin{rem}
	Because a general projection to $\P^m, m>n+1$ may be followed by a general projection
	$\P^m\cdots\to\P^{n+1}$ to yield a general projection to $\P^{n+1}$, the bounds
	of Corollary \ref{exp} also hold to projection to $\P^m$ for all $m\geq n+1$. But 
	these bounds need not be sharp for  $m>>n$.
\end{rem}
  
   \begin{proof}[Proof of Theorem \ref{contact-bound}]
 We work locally at an isolated point $p\in X\cap L$. We will 
 assume $m\leq c $ as the case $m>c$ is similar.
 We begin by representing $X$ locally as a graph over a submanifold of $L$,
 just as in \S \ref{order-projective}.
 Let $y_1,...,y_m$ be affine linear equations for $L$. Then, as
 discussed above, we may choose local analytic coordinates $x_1,...,x_{n+c-m}$
 on $L$ so that local analytic equations for $X$ have the form
 \eqspl{equations}{
 	y_1-f_1(x),...,y_m-f_m(x), f_{m+1}(x),...,f_c(x) ,}
 	with the $f_i$ functions of the $x$ coordinates only, vanishing at $p$,
 so that $( d_1,...,d_m)=\Ord_p(L, X)$ is an order
 sequence for $X, L, p$ with adapted basis $(f_\bullet)$.
 Here $f_{m+1},...,f_c$ cut out the projection
 $\bar X$ on $L$ which is smooth hence they
  have linearly independent differentials at
 $p$ that cut out the tangent space at $p$ to
 $\bar X$. NB $\bar X$ is not the projection of $X$
 from $\Lambda$ to $\P^m$.
 
 \par
 Now consider a local first-order deformation of $L$ over $\C[\epsilon]/(\epsilon^2)$.
 This corresponds to a subscheme
 $\mathscr L$ of $\P^{n+c}\times\Spec(\C[\epsilon]/(\epsilon^2))$, flat over $\Spec(\C[\epsilon]/(\epsilon^2))$
 and extending $L$, or equivalently to an element $v$ of the tangent space
 of the Grassmannian $\gg(n+c-m, n+c)$ at $[L]$, i.e. $v\in H^0(N_L)$.
Equation-wise, $\mathscr L$ is given near $p$  by deforming each equation
 $y_i$ to 
 \[y_i'=y_i+\epsilon g_i(x), i=1,...,m, \epsilon^2=0\]
 (for a global deformation $g_i$ is linear but this is local).
 This corresponds to the normal vector field $v$
 (section of $N_L=\Hom(\I_L, \O_L)$) defined by
  $(y_i\mapsto g_i(x), i=1,...,m)$.\par
  Now consider a first-order deformation $u$ of the point $p$
  in the ambient space $\P^{n+c}$. In coordinates $u$ is given by
  a pair $(b, a)$ where $b=(b_1,...,b_m)$ is the $y$-component and
  $a$ is the $x$ component, i.e. the projection to $L$.
  Compatibility of $u$ and $v$, i. e. the condition that the pair
  $(u,v)$ is tangent to the incidence locus of points on subspaces,
  reads in coordinate form\eqspl{u,v}{
  	b_i-g_i(0)=0.
  }
The condition that $u$ is tangent to $X$ reads
\eqspl{TX}{
	b_i-a\cdot\nabla f_i(0)=0, i=1,...,m\\
	a\cdot\nabla f_i(0)=0, i=m+1,...,c.
}
The condition on $a$ alone is then
\eqspl{TXa}{
	a\cdot\nabla f_i(0)=g_i(0), i=1,...,m\\
	a\cdot\nabla f_i(0)=0, i=m+1,...,c,
}
This then is exactly the condition that $a$ represent a tangent vector
at $p$ that is compatible with $(g_\bullet(x))$ and tangent to $X$.
Note that these conditions imply that $g_i(0)=0$ whenever $\ord_p(f_i)>1$.\par

Now assume $(L, p)$ is general with given order sequence
 $\Ord_p(L, X)$,  and consider a deformation of $(p,L)$ that preserves $\Ord_p(L, X)$.
 Then $(L, p)$ must deform to a nearby 
 pair $(L', p')$ having similar position with respect to $X$
 at $p'$,
 so $X$ is locally given by similar equations 
 \[y'_1-f'_1,..., y'_m-f'_m, f'_{m+1},...,f'_c\]
 as in \eqref{equations},  
 and because the order sequence is preserved, the $f'_i$ 
 must have the same order
 $d_i$ as $f_i$, albeit 
 with respect to $p'$.
  Now we have, since $\epsilon^2=0$ 
  (which implies $\epsilon g_i(x+\epsilon a)=\epsilon g_i(x)$),
  that for all $i\leq m$,
 \eqspl{nabla}{y'_i-f'_i(x+\epsilon a)=y_i+\epsilon g_i(x+\epsilon a)-f_i(x+\epsilon a)
 =y_i+\epsilon g_i(x)-f_i(x)-\epsilon a\cdot\nabla f_i(x)}
 where $\ord_p(\nabla f_i)\geq d_i-1$.
 Because $y'_i-f'_i(x+\epsilon a)$ must have order $d_i$ at least,
 the terms of order $<d_i$ in $g_i(x)$ must cancel out with the
 $\epsilon a\cdot\nabla f_i(x)$, while $\ord_p(\nabla f_i)\geq d_i-1$
 it follows that $g_i(x)$ cannot have any term of order
 $<d_i-1$, in other words 
 \eqspl{order}{\ord_p(g_i)\geq d_i-1, \forall  i\leq m}
 is a \emph{necessary} condition on the deformation given by the $g_i$
 to preserve the order sequence. \par
 We denote the subsheaf of $N_L$ defined by the conditions \eqref{order} 
 for given $p$ by $N_{L,p}^{\orp}$
  and set $N_L^{\orp}=\bigcap\limits_p N_{L, p}^\orp$.
  This is called the \emph{order-preserving} subsheaf.
	\begin{rem}[Incidental remark]\label{sufficiency} Though unimportant for the proof, one may wonder, 
what about sufficiency ?
 Indeed sufficiency is \emph{irrelevant} for the purpose of proving the Theorem:
 all we need is that there is \emph{some} subsheaf,
say $N^1_L$  of $N_L^\ord$,
 containing $N^{\ord, 1}$, such that a local first-order deformation of
 $L$ preserves order sequence iff it is locally in $N^1_L$ near each 
 $p\in L\cap X$. As the argument below will show, the exact nature of $N^\orp_L$
is immaterial. That said, note as to the sufficiency question that by
\eqref{nabla}, 
given $g_i(x)$ of order $\geq d_i-1$, the function
 $y'_i-f'_i(x+\epsilon a)$ as above will have
 order $d_i$ for given $a$ whenever the term of order $d_i-1$
 in $g_i$ is cancelled by the like term in
  $a\cdot\nabla f_i(x)$. Here $a$ must satisfy the
 conditions \eqref{TXa}, meaning that the point $p$
 is moving compatibly with the motion of $L$ and remains on $X$ (which is not moving). 
 So given $(g_\bullet(x))$, the sufficient condition
 at $p$ to preserve order sequence is that there should exist $a$ satisfying
the conditions \eqref{TXa} such that
\[(g_i(x))\equiv (a\cdot\nabla f_i(x))\mod m_p^{d_i}.\] 
Again, as we saw, the conditions on $a$ 
mean that $a$ is the $L$-projection of a tangent vector at $p$ that is
compatible with the deformation $(g_\bullet(x))$ and tangent to $X$.\qed
\end{rem}
Returning to the main argument, note the exact diagram
\eqspl{secant-seq}{
\begin{matrix}
	&&&&&0&\\
	&&&&&\uparrow&\\
	0\to&N_L^\ord&\to&N_L&\to&A^\ord&\to 0\\
	&\uparrow&&\parallel&&\uparrow&\\
	0\to&N_L^\orp&\to&N_L&\to& A^\orp&\to 0\\
	&\uparrow&&&&&\\
	&0&&&&&
	\end{matrix}	
}
 where $A^\ord, A^\orp$ are torsion sheaves supported at $L\cap X$.
  Also $A^\ord$ decomposes as direct sum
 of cyclic torsion sheaves, of total length at least  $a(X, L)$.
 Now note the following:
 \begin{lem}[Generalized vanishing lemma]\label{gvan-lem}
 	Suppose $(\Lambda, L, \P^M)\simeq (\P^{M-m-1}, \P^{M-m}, \P^M)$,
 	and let $N^0\subset N_L$ be a modification on a finite subset
 	disjoint from $\Lambda$. Assume the image of
 	\eqspl{rest1}
 	{H^0(L, N^0)\to H^0(\Lambda, N^0|_\Lambda)=H^0(\Lambda, N_L|_\Lambda)
 	}
 	contains the image of
 	\eqspl{rest2}
 	{H^0(\Lambda, N_\Lambda)\to H^0(\Lambda, N_L|_\Lambda).
 	}
 	Then $H^1(L, N^0(-1))=0$.
 	\end{lem}
\begin{proof} The proof is identical to the proof of  
 the Vanishing Lemma (\cite{filling}, Lemma 5.8).
 \end{proof}
\begin{rem*}
	In fact the map \eqref{rest2} is clearly surjective, so the assumption
	is actually equivalent to  the map \eqref{rest1} being surjective.
	Intuitively the assumption as stated means that any given 1-parameter
	deformation of $\Lambda$ is 'covered' by a compatible 
	$N^0$ deformation of $L$, where compatibility means having the same
	image in $H^0(\Lambda, N_L|_\Lambda)$ (note that $N_L|_\Lambda$
	is a quotient of $N_\Lambda$).\qed
\end{rem*}
Now as in \cite{filling}, 
using $\mathrm{char}(\C)=0$,
the hypothesis that $\Lambda$ is general with
respect to $X$ and that $L$ is general with given order sequences at 
each point $p\in L\cap X$ shows that the hypotheses of Lemma \ref{gvan-lem}
are satisfied for $N^{\orp}_L$ in place of $N^0$. Briefly, this results from
the fact that the scheme $\mathcal Z$
parametrizing triples $z=(\Lambda, L, p)$ 
with given order sequence- in fact already the corresponding reduced 
scheme $\mathcal Z_{\mathrm{red}}$-
projects generically surjectively to the Grassmannian 
$G=\mathbb G(\lambda, n+c)$  parametrizing $\Lambda$s,
where the map $\mathcal Z\to G$ is the composite of
the natural map of $\mathcal Z$ to the flag variety of pairs $(\Lambda, L)$
and the projection of the latter to $G$.
Then by generic smoothness in char. 0 this induces a surjection on
Zariski tangent spaces $T_z\mathcal Z_\mathrm{red}\to T_\Lambda G$,
a fortiori $T_z\mathcal Z\to T_\Lambda G=H^0(\Lambda, N_\Lambda)$ is surjective.
But by definition of order preserving and generality of $L$ with given order sequence,
we have a diagram
\eqspl{}{
\begin{matrix}
	T_z\mathcal Z&\to &H^0(\Lambda, N_\Lambda)\\
	\downarrow&&\downarrow\\
	H^0(N^\orp)&\to&H^0(N_L|_\Lambda)
	\end{matrix}	
}
where the right vertical arrow is obviously surjective. Thus, the hypotheses of Lemma
 \ref{gvan-lem} are satisfied.
Consequently by the Lemma, 
the cohomology sequence of the bottom row of 
\eqref{secant-seq} twisted by -1 yields
 \[m=h^0(N_L(-1))\geq\l(A^\orp)\geq \l(A^\ord)\geq a(X, L).\]\par
 This proves the theorem. 
 \end{proof}
\begin{example} This is an example of a 1-st order deformation preserving 
	order sequence.
	Consider the situation of Example \ref{3dim}. Deforming $L$ by 
	$\epsilon g_1=6\epsilon x^5, \epsilon g_2=3\epsilon x^2$, 
	i.e. deforming to $y_1=6\epsilon x^5, y_2=3\epsilon x^2$,
	the deformation
	of $L$ has the same order sequence
	with adapted basis $((x+\epsilon)^6, (x+\epsilon)^3)$
	at the point on $X$ with $x$-coordinate $-\epsilon$,
	i.e. $(-\epsilon, 0, 0)$, which is a deformation
	of the origin on $X$. 
	\end{example}
   \bibliographystyle{amsplain}
   \bibliography{./mybib}

\providecommand{\bysame}{\leavevmode\hbox to3em{\hrulefill}\thinspace}
\providecommand{\MR}{\relax\ifhmode\unskip\space\fi MR }
\providecommand{\MRhref}[2]{%
  \href{http://www.ams.org/mathscinet-getitem?mr=#1}{#2}
}
\providecommand{\href}[2]{#2}
\begin{thebibliography}{10}

\bibitem{stackexchange}
User~Andreas 677727, \emph{Max of an exponential sum}, Stackexchange (2019),
  \url{https://math.stackexchange.com/q/3280075} (sourced 2019-07-23).

\bibitem{alzati-ottaviani}
A.~Alzati and G.~Ottaviani, \emph{The theorem of {M}ather on generic
  projections in the setting of algebraic geometry}, Man. math. \textbf{74}
  (1992), 391--412.

\bibitem{beis}
R.~Beheshti and D.~Eisenbud, \emph{Fibers of generic projections}, Comp. math
  \textbf{146} (2010), 435--456.

\bibitem{eisenbud-harris}
D.~Eisenbud and J.~Harris, \emph{Limit linear series: Basic theory},
  Inventiones math. \textbf{85} (1986), 337--371.

\bibitem{griffiths-harris}
Phillip Griffiths and Joseph Harris, \emph{Principles of algebraic geometry},
  Pure and applied mathematics, a Wiley-Interscience series of texts,
  monographs and tracts, John Wiley and sons, 1978.

\bibitem{grp}
L.~Gruson and C.~Peskine, \emph{On the smooth locus of aligned {H}ilbert
  schemes. {T}he $k$-secant lemma and the general projection theorem}, Duke
  math. J. \textbf{162} (2013), 553--578, \url{arxiv.org /1010.2399}.

\bibitem{laz-pos}
Robert Lazarsfeld, \emph{Positivity in algebraic geometry}, Springer, 2004.

\bibitem{mather}
J.~Mather, \emph{Generic projections}, Ann. Math. \textbf{98} (1973), 226--245.

\bibitem{filling}
Z.~Ran, \emph{Unobstructedness of filling secants and the {G}ruson-{P}eskine
  general projection theorem}, Duke math. J. \textbf{164} (2015), 739--764,
  \url{arxiv.org/1302.0824}.

\bibitem{fillingreduced}
\bysame, \emph{Superficial fibres of generic projections}, J. Inst. Math.
  Jussieu \textbf{17} (2018), no.~2, 419--425, doi:10.1017/S1474748016000050.

\end{thebibliography}
   
\end{document}